\documentclass[12pt]{article}
\usepackage[latin1]{inputenc}
\usepackage[T1]{fontenc}
\usepackage{amsmath,amssymb,amsfonts,amsxtra,amsthm}
\usepackage{mathrsfs}
\usepackage{dsfont}
\usepackage{color}
   \definecolor{cites}{rgb}{0.75 , 0.00 , 0.00}  
   \definecolor{urls} {rgb}{0.00 , 0.00 , 1.00}  
   \definecolor{links}{rgb}{0.00 , 0.00 , 0.5}   
  \definecolor{gray}{rgb}{0.5,0.5,.5}
\usepackage[plainpages=false,colorlinks=true, citecolor=cites, urlcolor=urls, linkcolor=links, breaklinks=true]{hyperref}

\newtheorem{theorem}{Theorem}
\newtheorem{lemma}[theorem]{Lemma}
\newtheorem{proposition}[theorem]{Proposition}
\newtheorem{corollary}[theorem]{Corollary}

\theoremstyle{definition}
\newtheorem{definition}[theorem]{Definition}

\theoremstyle{remark}
\newtheorem{remark}[theorem]{Remark}

\numberwithin{equation}{section}

\allowdisplaybreaks

\newcommand{\B}{\mathbb{B}}
\newcommand{\C}{\mathbb{C}}
\newcommand{\D}{\mathbb{D}}
\newcommand{\N}{\mathbb{N}}
\newcommand{\1}{\mathds{1}}

\newcommand{\Ac}{\mathcal{A}}
\newcommand{\Lc}{\mathcal{L}}

\renewcommand{\epsilon}{\varepsilon}

\DeclareMathOperator{\dist}{dist}
\DeclareMathOperator{\supp}{supp}
\DeclareMathOperator{\VMO}{VMO}

\newcommand{\from}{\colon}

\providecommand{\scpr}[2]{\left\langle #1, #2 \right\rangle}
\renewcommand{\sp}{\scpr}
\providecommand{\abs}[1]{\left\lvert#1\right\rvert}
\providecommand{\norm}[1]{\left\lVert#1\right\rVert}
\providecommand{\set}[1]{\left\{ #1\right\}}

\begin{document}

\title{Compact Hankel operators with bounded symbols}
\author{Raffael Hagger and Jani A.~Virtanen}

\date{}
\maketitle

\makeatletter
\def\blfootnote{\gdef\@thefnmark{}\@footnotetext}
\makeatother

\begin{abstract} We give a new proof of the result that the Hankel operator $H_f$ with a bounded symbol is compact on standard weighted Fock spaces $F^2_\alpha(\C^n)$ if and only if $H_{\bar f}$ is compact. Our proof uses limit operator techniques and extends to $F^p_\alpha(\C^n)$ when $1<p<\infty$. It also fully explains that this striking result is caused by the lack of nonconstant bounded analytic functions in the complex plane unlike in Bergman spaces $A^p_\alpha(\Omega)$. Furthermore, we show that the compactness of Hankel operators is independent of $p$ and $\alpha$ for both $F^p_\alpha(\C^n)$ and $A^p_\alpha(\Omega)$, where $\Omega$ is a bounded symmetric domain in $\C^n$.

\medskip


\end{abstract}

\maketitle

\section{Introduction}

In this note we study Hankel operators on Bergman and Fock spaces simultaneously. For simplicity, we therefore denote both the standard weighted Bergman spaces of bounded symmetric domains $\Omega$ and the standard weighted Fock spaces of $\Omega = \C^n$ by $X^p_\alpha(\Omega)$ (see Section~\ref{preliminaries} for the precise definitions of these spaces). The ambient space will always be denoted by $L^p_{\alpha}(\Omega)$. Let $P_{\alpha}$ be the orthogonal projection of $L^2_\alpha(\Omega)$ onto $X^2_\alpha(\Omega)$. The operator $P_{\alpha}$ is usually called the Bergman projection and extends boundedly to a projection $L^p_\alpha(\Omega) \to X^p_\alpha(\Omega)$ in all cases we will consider here. The Hankel operator $H_f : X^p_\alpha(\Omega) \to L^p_\alpha(\Omega)$ is defined by
\begin{equation*}
H_f g = Q_{\alpha}(fg),
\end{equation*}
where $Q_{\alpha} = I-P_{\alpha}$ is the complementary projection of $P_{\alpha}$.

Many classical results about Hankel operators, such as the Nehari and Hartman theorems, have been obtained first in the setting of the Hardy space before their treatment in Bergman and Fock spaces; see \cite{MR2223704,MR1949210}. Regarding compactness, while there are similarities, there are also striking differences between the properties of Hankel operators on these three function spaces. One difference is that in the case of Hardy spaces most results are only available in dimension $n=1$, while for Hankel operators on Bergman and Fock spaces, most results are known for $n\geq 1$, as in our present work. Perhaps most importantly, many of the differences are explained by the size of the complement of the domain in $L^p$; e.g. the complement of $A^p$ is much larger than $A^p$ itself, while $H^p$ and its complement are of the same size.

In 1984, Axler posed the question of characterizing compact Hankel operators on the unweighted Bergman space $A^2(\D)$ of the open unit disk $\D$, and two years later, in \cite{MR0850538}, he showed that when $f\in A^2(\D)$, the Hankel operator $H_{\bar f}$ (with a conjugate analytic symbol) is compact if and only if $f$ is in the little Bloch space; that is, $(1-|z|^2)f'(z) \to 0$ as $|z|\to 1$. 

To state results for general bounded symbols, for each $\lambda\in\Omega$, denote by $\phi_\lambda$ the automorphism of $\Omega$ with the properties that $\phi_\lambda (\lambda) = 0$ and $\phi_\lambda\circ \phi_\lambda$ is the identity map on $\Omega$. For $\Omega = \D$ these are the usual M\"obius transformations
\begin{equation*}
\phi_\lambda(z) = \frac{\lambda-z}{1-\bar\lambda z} \quad (z \in \D).
\end{equation*}
In the early 1990s, Stroethoff and Zheng, independently for the unit disk $\Omega = \D$ and jointly for the unit ball $\Omega=\B_n$, characterized Hankel operators with bounded symbols $f\in L^\infty(\Omega)$ (see~\cite{StroeZhe} and the references therein). In particular, they showed that the Hankel operator $H_f$ is compact on $A^2(\Omega)$ if and only if for some (or equivalently, for all) $1\leq q<\infty$, we have
\begin{equation}\label{e:Stroethoff condition}
	\|Q(f\circ \phi_\lambda)\|_q \to 0
\end{equation}
as $\lambda\to \partial \Omega$.

In 1987, Zhu~\cite{MR1178032} proved that $H_f$ and $H_{\bar f}$ are simultaneously compact on $A^p_\alpha(\B_n)$ if and only if $f\in \VMO^p_r(\B_n)$, that is,
\begin{equation}\label{e:VMOp}
	\lim_{z\to \partial \B_n} \frac1{|B(z,r)|} \int_{B(z,r)} |f(w) - \hat f_r(z)|^p\, \mathrm{d}v(z) = 0,
\end{equation}
where $\hat f_r(z)$ denotes the Euclidean average of $f$ over the Bergman ball $B(z,r)$ and $|B(z,r)|$ denotes the Euclidean volume of $B(z,r)$. It is well known that $\VMO^p_r(\B_n)$ is independent of $r$ (see~\cite{MR1178032}); we set $\VMO^p(\B_n) = \VMO^p_1(\B_n)$. A generalization of Zhu's result to bounded symmetric domains when $p=2$ can be found in~\cite{MR1073289}. It is worth noting here that there are bounded symbols $f$ for which $H_f$ is compact on $A^p_\alpha(\B_n)$ but $H_{\bar f}$ is not compact. An interesting example is a Blaschke product $b$ with zeros at $\alpha_k = 1-1/2^k$; that is,
\begin{equation*}
b(z) = \prod_{k=1}^\infty \frac{\alpha_k - z}{1-\bar \alpha_k z} \frac{|\alpha_k|}{\alpha_k}
\end{equation*}
for $z\in \D$. The function $b$ is not in the little Bloch space (see~\cite{MR1038664}) and hence $H_{\bar b}$ is not compact by Axler's result above. However, $H_b = 0$ because $b$ is a bounded analytic function in $\D$. 

We now compare the situation with Hankel operators on Fock spaces. The result on the simultaneous compactness of $H_f$ and $H_{\bar f}$ is the same as in Bergman spaces, which was first proven for Hankel operators on $F^2_\alpha(\C^n)$ by Bauer~\cite{MR2138695} in 2005. For an extension to $1<p<\infty$, see~\cite{L18, PeSchuVi}. Here $\VMO^p(\C^n)$ is defined analogously to $\VMO^p(\Omega)$ by replacing the Bergman balls with Euclidean balls (see~\cite{PeSchuVi}). Regarding~\eqref{e:Stroethoff condition}, Stroethoff~\cite{Stroethoff} showed that the same condition is both sufficient and necessary for $H_f$ to be compact on $F^2_\alpha(\C^n)$ (but proved only for $q=2$ unlike in the Bergman space where the general case $1<q<\infty$ is used for subsequent results). 

What is very different about Hankel operators on Fock spaces is that, for bounded symbols, $H_f$ is compact if and only if $H_{\bar f}$ is compact. This was proven for $p=2$ by Berger and Coburn~\cite{MR0882716} in 1987 and by Stroethoff~\cite{Stroethoff} in 1992 using more elementary methods. As for the reason for this difference, Zhu~\cite{MR2934601} recently commented: ``{\it A partial explanation for this difference is probably the lack of bounded analytic or harmonic functions on the entire complex plane.}'' This naturally suggests that Berger and Coburn's result should remain true for the other Fock spaces $F^p_\alpha$ with $1<p<\infty$. However, the previously used techniques seem unsuitable when $p\neq 2$. Let us also mention that there are several characterizations of compact Hankel operators with more general $L^1$-symbols due to Luecking and many others, but we will not discuss them here as their results and methodologies are rather function-theoretic and also not suitable for our purposes.

Several compactness characterizations for general bounded linear operators on Bergman and Fock spaces are available in the literature. These usually involve the Berezin transform or the normalized reproducing kernels $k_z \in X^2_\alpha$ (see Section \ref{preliminaries} for definitions). For example, Su\'arez et al.~\cite{MiSuWi,Suarez} showed that a bounded linear operator $A$ on $A^p_\alpha(\B_n)$ is compact if and only if it is contained in the Toeplitz algebra and its Berezin transform vanishes at the boundary. Subsequently, similar results were obtained for Fock spaces \cite{BaIs,IMW15} and Bergman spaces over different domains such as the polydisk \cite{MW14}, bounded symmetric domains \cite{Hagger2}, the Thullen domain \cite{HW20} or general Bergman type function spaces \cite{MiWi}. These results do not directly apply to Hankel operators $H_f: X^p_{\alpha} \to L^p_{\alpha}$. However, for $p = 2$ there is a way around this by considering $H_f^*H_f$. This leads to the necessary condition $\|H_fk_z\|_{L^2_\alpha} \to 0$ as $z$ tends to the boundary of $\Omega$ (compare with Theorem \ref{thm1}). Our approach is more direct and also works for arbitrary $1<p<\infty$. It unifies the two cases, Bergman and Fock, and clearly explains the above mentioned difference between the two families of spaces, confirming Zhu's conjecture. More precisely, we will give an alternate proof of the Berger-Coburn theorem using limit operator techniques that works for all Fock spaces $F^p_\alpha(\C^n)$ with $1<p<\infty$ and fully explains the difference between Bergman and Fock spaces (see Theorems \ref{thm2} and \ref{thm3}). Namely, it will become apparent that the only ingredient missing for the same proof to work for Hankel operators on Bergman spaces is Liouville's theorem. In Theorem~\ref{thm1} and Corollary~\ref{cor:thm1_2}, we also show that the compactness of $H_f : X^p_\alpha(\Omega) \to L^p_\alpha(\Omega)$ is independent of $p$ and $\alpha$, and hence generalize the results of Stroethoff and Zheng which state that~\eqref{e:Stroethoff condition} is sufficient and necessary for $H_f : X^2_\alpha(\Omega) \to L^2_\alpha(\Omega) $ to be compact.

\section{Preliminaries}\label{preliminaries}

Throughout this paper let $1<p<\infty$ and $n\in\N$. For $\alpha>0$, let $\mu_{\alpha}$ be the Gaussian measure defined on $\C^n$ by
\begin{equation}\label{e:Fock measure}
	\mathrm{d}\mu_{\alpha}(z) := \left(\frac{\alpha}{\pi}\right)^n \mathrm{e}^{-\alpha\abs{z}^2} \, \mathrm{d}v(z),
\end{equation}
where $v$ is the usual Lebesgue measure on $\C^n$, and set $L^p_{\alpha} := L^p(\C^n,\mathrm{d}\mu_{p\alpha/2})$. For $f \in L^p_{\alpha}$ and $z \in \C^n$, we define the weighted shift $C_z$ by
\begin{equation*}
(C_zf)(w) = f(\phi_z(w)) \, \mathrm{e}^{\alpha\sp{w}{z}-\frac{\alpha}{2}\abs{z}^2},
\end{equation*}
where $\phi_z(w) = w-z$ for $w\in \C^n$. It is easy to check that $C_z$ is a surjective isometry with $C_z^{-1} = C_{-z}$ and $C_zM_fC_z^{-1} = M_{f \circ \phi_z}$ for every multiplication operator $M_f$ with bounded symbol $f$. The \emph{Fock space} $F^p_{\alpha}$ is the closed subspace of all analytic functions in $L^p_{\alpha}$. The orthogonal projection $P_{\alpha}$ of $L^2_{\alpha}$ onto $F^2_{\alpha}$ is given by
\begin{equation*}(P_{\alpha}f)(z) = \int_{\C^n} f(w) \, \mathrm{e}^{\alpha\sp{z}{w}} \, \mathrm{d}\mu_{\alpha}(w)
\end{equation*}
and it extends to a bounded projection of $L^p_\alpha$ onto $F^p_\alpha$. $K_z(w) := \mathrm{e}^{\alpha\sp{w}{z}}$ is called the reproducing kernel. As
\begin{equation*}
\|K_z\|_{L^p_{\alpha}} = \mathrm{e}^{\frac{\alpha}{2}|z|^2},
\end{equation*}
we have $K_z \in F^p_{\alpha}$ for all $p \in (1,\infty)$, $\alpha > 0$ and $z \in \C^n$. We can therefore define the normalized reproducing kernel $k_z$ by
\begin{equation*}
k_z(w) =  \mathrm{e}^{\alpha\sp{w}{z}-\frac{\alpha}{2}|z|^2}.
\end{equation*}
Note that $C_z\1 = k_z$, where $\1$ denotes the constant function $1$. For $f \in L^{\infty}(\C^n)$, the Berezin transform of $f$ is defined as 
\begin{equation*}
\tilde{f}(z) = \sp{fk_z}{k_z}_{L^2_{\alpha}} = \int_{\C^n} (f \circ \phi_{-z})(w) \, \mathrm{d}\mu_{\alpha}(w).
\end{equation*}

Now let $\alpha>-1$ and suppose that $\Omega\subset \C^n$ is an irreducible bounded symmetric domain in its Harish-Chandra realization (see~\cite{Englis, Hagger2, Upmeier}). We define
\begin{equation*}
\mathrm{d}\mu_{\alpha}(z) := c_{\alpha} h(z,z)^{\alpha} \, \mathrm{d}v(z),
\end{equation*}
where $h$ is the so-called Jordan triple determinant (see e.g.~\cite{Englis,Hagger2,Upmeier}), and $c_{\alpha}$ is a constant such that $\mu_{\alpha}(\Omega) = 1$. For the unit ball $\Omega = \B_n$ we have $h(z,w) = 1 - \sp{z}{w}$ and $c_{\alpha} = \frac{\Gamma(n+\alpha+1)}{n!\Gamma(\alpha+1)}$. For $p \in (1,\infty)$ satisfying
\begin{equation} \label{p_constraint}
1 + \frac{(r-1)a}{2(\alpha+1)} < p < 1 + \frac{2(\alpha+1)}{(r-1)a}
\end{equation}
we set $L^p_{\alpha} := L^p(\Omega,\mathrm{d}\mu_{\alpha})$. Here, $r$ and $a$ are two constants depending on the domain $\Omega$ (see \cite{Englis} for a table and \cite{Hagger2} for a discussion of \eqref{p_constraint}). For $\Omega = \B_n$, it holds $r = 1$ and $a = 2$, hence every $p \in (1,\infty)$ is permitted.

For $f \in L^p_{\alpha}$ and $z \in \Omega$, we define the reflection $C_z$ by
\begin{equation*}
(C_zf)(w) = f(\phi_z(w))\frac{h(z,z)^{\frac{\alpha+g}{p}}}{h(w,z)^{\frac{2(\alpha+g)}{p}}},
\end{equation*}
where $g$ is another constant depending on $\Omega$ ($g = n+1$ for $\Omega = \B_n$) and $\phi_z$ is the (unique) geodesic symmetry interchanging $0$ and $z$ (M\"obius transforms in case $\Omega = \B_n$). It is not difficult to check that $C_z$ is a surjective isometry with $C_z^{-1} = C_z$ (observe the difference between Bergman and Fock spaces here) and $C_zM_fC_z^{-1} = M_{f \circ \phi_z}$ for every multiplication operator $M_f$ with bounded symbol $f$. The \emph{Bergman space} $A^p_{\alpha}$ is the closed subspace of all analytic functions in $L^p_{\alpha}$. The orthogonal projection $P_{\alpha}$ of $L^2_{\alpha}$ onto $A^2_{\alpha}$ is given by
\begin{equation*}
(P_{\alpha}f)(z) = \int_{\Omega} f(w)h(z,w)^{-\alpha-g} \, \mathrm{d}\mu_{\alpha}(w)
\end{equation*}
and it extends to a bounded projection of $L^p_\alpha$ onto $A^p_\alpha$ (provided that $p$ satisfies \eqref{p_constraint}). The projection $P_\alpha$ is referred to as the \emph{Bergman projection} and $K_z(w) := h(w,z)^{-\alpha-g}$ is called the reproducing kernel. It is again not difficult to check that $K_z \in A^p_{\alpha}$ for all $\alpha > -1$, $p \in (1,\infty)$ that satisfy \eqref{p_constraint} and $z \in \Omega$. However, in this case the norm of $K_z$ depends on $p$. We therefore define
\begin{equation*}
k_z(w) := \frac{K_z(w)}{\|K_z\|_{L^2_{\alpha}}} = \frac{h(z,z)^{\frac{\alpha+g}{2}}}{h(w,z)^{\alpha+g}}
\end{equation*}
and note $C_z\1 = k_z^{2/p}$. For $f \in L^{\infty}(\C^n)$, the Berezin transform of $f$ is again defined as
\begin{equation*}
\tilde{f}(z) = \sp{fk_z}{k_z}_{L^2_{\alpha}} = \int_{\Omega} (f \circ \phi_z)(w) \, \mathrm{d}\mu_{\alpha}(w).
\end{equation*}

When the results and their proofs are similar, we denote both  $A^p_\alpha$ and $F^p_\alpha$ by $X^p_\alpha$ and their corresponding domains by $\Omega$, that is,
\begin{equation*}
X^p_\alpha \in \{ A^p_\alpha, F^p_\alpha\}.
\end{equation*}
When $X^p_\alpha = A^p_\alpha$, it is understood that $\alpha>-1$ satisfies \eqref{p_constraint} and $\Omega$ is a bounded symmetric domain in $\C^n$, while for $X^p_\alpha = F^p_\alpha$, $\alpha>0$ and $\Omega = \C^n$, in which case $\partial \Omega$ denotes the point at infinity. Finally,  $d_{\Omega}$ will denote the Bergman metric for bounded symmetric domains and the Euclidean metric for $\Omega = \C^n$. The corresponding open balls are denoted by
\begin{equation*}
B(z,r) := \set{w \in \Omega : d_{\Omega}(w,z) < r}
\end{equation*}
for $z \in \Omega$ and $r > 0$.

For $f\in L^\infty(\Omega)$, we define the \emph{Hankel operator} $H_f : X^p_\alpha \to L^p_\alpha$ by
\begin{equation*}
H_f g = (I-P_\alpha)(fg)
\end{equation*}
and the \emph{Toeplitz operator} $T_f : X^p_\alpha \to X^p_\alpha$ by
\begin{equation*}
T_fg = P_\alpha(fg).
\end{equation*}
The following Banach algebra plays an important role in our analysis:
\begin{equation*}
\Ac := \set{f \in L^{\infty}(\Omega) : H_f \text{ is compact}}.
\end{equation*}

Our first auxiliary result is the following lemma, which demonstrates the role of $C_z$ in connection with compactness and plays an important role in the proofs of our main results. It is very similar to earlier results like \cite[Theorem 6.2]{BaIs}, \cite[Proposition 19]{Hagger2} or \cite[Theorem 9.3]{Suarez}. The only difference is that we need it for operators from $X^p_{\alpha}$ to $L^p_{\alpha}$.

\begin{lemma} \label{lem1}
If $H \from X^p_{\alpha} \to L^p_{\alpha}$ is a compact operator, then $\norm{C_zHC_z^{-1}g} \to 0$ for every $g \in X^p_{\alpha}$ as $z \to \partial\Omega$.
\end{lemma}

\begin{proof}
Let $g \in X^p_{\alpha}$. Then
\begin{align*}
\norm{C_zHC_z^{-1}g} &\leq \norm{C_zHP_{\alpha}M_{\chi_{B(0,r)}}C_z^{-1}g} + \norm{C_zHP_{\alpha}M_{\1-\chi_{B(0,r)}}C_z^{-1}g}\\
&\leq \norm{HP_{\alpha}}\norm{M_{\chi_{B(0,r)}}C_z^{-1}g} + \norm{HP_{\alpha}M_{\1-\chi_{B(0,r)}}}\norm{g}\\
&= \norm{HP_{\alpha}}\norm{M_{\chi_{B(z,r)}}g} + \norm{HP_{\alpha}M_{\1-\chi_{B(0,r)}}}\norm{g},
\end{align*}
where we used $C_zM_{\chi_{B(0,r)}}C_z^{-1} = M_{\chi_{B(z,r)}}$ and the fact that $C_z$ and $C_z^{-1}$ are both isometries. Since $\chi_{B(z,r)} \to 0$ pointwise as $z \to \partial\Omega$, the first term tends to $0$ by dominated convergence. Similarly, $M_{\1-\chi_{B(0,r)}}$ tends strongly to $0$ as $r \to \infty$. As $H$ is assumed to be compact, this implies $\norm{HP_{\alpha}M_{\1-\chi_{B(0,r)}}} \to 0$ as $r \to \infty$. Hence, choosing $r$ sufficiently large, we can assume that $\norm{HP_{\alpha}M_{\1-\chi_{B(0,r)}}}$ is arbitrarily small. We conclude that $\norm{C_zHC_z^{-1}g} \to 0$ as $z \to \partial\Omega$.
\end{proof}

Let $\beta\Omega$ denote the Stone-\v{C}ech compactification of $\Omega$. By its universal property, any continuous map $f$ from $\Omega$ to a compact Hausdorff space $K$ can be uniquely extended to a continuous map $f \from \beta\Omega \to K$. Here, we do not distinguish between $f$ and its extension to $\beta\Omega$. Note that $\beta\Omega$ can be realized as the maximal ideal space of bounded continuous functions defined on $\Omega$. Every maximal ideal corresponds to a point in $\beta\Omega$ via evaluation.

The following proposition will be proven after Remark~\ref{rem:compactifications} and Lemma~\ref{lem:lo_ex} below.

\begin{proposition} \label{prop2}
Let $f \in \Ac$ and $x \in \beta\Omega \setminus \Omega$. Then there is a bounded analytic function $h_x$ such that for all nets $(z_{\gamma})$ in $\Omega$ converging to $x$:
\begin{itemize}
	\item[\rm (i)] $\norm{f \circ \phi_{z_{\gamma}} - h_x}_{L^p_{\alpha}} \to 0$ as $z_{\gamma} \to x$,
	\item[\rm (ii)] $C_{z_{\gamma}}M_fC_{z_{\gamma}}^{-1} = M_{f \circ \phi_{z_{\gamma}}}$ converges strongly to $M_{h_x}$,
	\item[\rm (iii)] $C_{z_{\gamma}}M_{\bar{f}}C_{z_{\gamma}}^{-1} = M_{\bar{f} \circ \phi_{z_{\gamma}}}$ converges strongly to $M_{\overline{h_x}}$.
	\end{itemize}
\end{proposition}

\begin{remark} \label{rem:compactifications}
In the literature two different compactifications are used to achieve more or less the same thing, namely the Stone-\v{C}ech compactification e.g.~in \cite{Hagger2,MiWi,Stroethoff} and the maximal ideal space of bounded uniformly continuous functions e.g.~in \cite{BaIs,Hagger1,MiSuWi}. Usually, this is just a matter of labeling limit operators. More precisely, if there are two compactifications of $\Omega$, say $\hat{\Omega}$ and $\tilde{\Omega}$, and a net $(z_{\gamma})$ in $\Omega$ converging to some $x \in \hat{\Omega}$, then by compactness there is a subnet, again denoted by $(z_{\gamma})$, such that $(z_{\gamma})$ also converges in $\tilde{\Omega}$. For an arbitrary operator $A$ the convergence of the corresponding net $(C_{z_{\gamma}}AC_{z_{\gamma}}^{-1})$, by definition, does not depend on the chosen compactification. The set of all limits of nets of the form $(C_{z_{\gamma}}AC_{z_{\gamma}}^{-1})$ is therefore the same for either compactification, namely exactly the closure of $\set
 {z \in \Omega : C_zAC_z^{-1}}$ in the strong operator topology. In fact, since bounded sets are metrizable in the strong operator topology, one may even take sequences instead of nets. However, it is convenient to label the limits in terms of boundary elements of the compactification. For this to make sense, for every net $(z_{\gamma})$ converging to the same $x \in \hat{\Omega} \setminus \Omega$, the limit of $(C_{z_{\gamma}}AC_{z_{\gamma}}^{-1})$ needs to be the same. For the Stone-\v{C}ech compactification this is rather easy to show. One only needs to show that $z \mapsto C_zAC_z^{-1}$ is weakly continuous as it may then be continuously extended to $\beta\Omega$, which implies the uniqueness.
\end{remark}

We will need the following lemma, which is a corollary of \cite[Proposition 14]{Hagger2} and \cite[Proposition 5.3]{BaIs}:

\begin{lemma} \label{lem:lo_ex}
Let $f \in L^{\infty}(\Omega)$ and $g \in X^p_{\alpha}$. Then the map $z \mapsto C_zT_fC_z^{-1}g$ extends continuously to $\beta\Omega$.
\end{lemma}

\begin{proof}
For bounded symmetric domains this is shown in \cite[Proposition 14]{Hagger2}. In the case $\Omega = \C^n$, \cite[Proposition 5.3]{BaIs} proves the result for the maximal ideal space of bounded uniformly continuous functions instead of $\beta\C^n$. Hence the result is obtained via Remark \ref{rem:compactifications}. To illustrate the argument, we provide some more details here. Let $\tilde{\Omega}$ denote the maximal ideal space of bounded uniformly continuous functions on $\C^n$. In \cite[Proposition 5.3]{BaIs} it is shown that $z \mapsto C_zT_fC_z^{-1}$ is strongly continuous on $\Omega$. In particular, it is weakly continuous. Moreover, it is bounded by $\norm{P_{\alpha}}\norm{f}_{\infty}$. As bounded sets are relatively compact in the weak operator topology, $z \mapsto C_zT_fC_z^{-1}$ can be extended to a weakly continuous map on $\beta\Omega$. It remains to show that $z \mapsto C_zT_fC_z^{-1}$ is also strongly continuous on $\beta\Omega$. Indeed, choose a net $(
 z_{\gamma})$ in $\Omega$ that converges to some $x \in \beta\Omega$. As $\tilde{\Omega}$ is compact, every subnet of $(z_{\gamma})$ has a subnet, again denoted by $(z_{\gamma})$ converging in $\tilde{\Omega}$. For each of these subnets the corresponding subnet $(C_{z_{\gamma}}T_fC_{z_{\gamma}}^{-1})$ converges strongly by \cite[Proposition 5.3]{BaIs}. The weak continuity implies that all these limits are the same. Hence the whole net converges strongly and thus $z \mapsto C_zT_fC_z^{-1}$ is strongly continuous on $\beta\Omega$.
\end{proof}

\begin{proof}[Proof of Proposition \ref{prop2}]
For $g \in X^p_{\alpha}$, we have
\begin{equation}
C_zM_fC_z^{-1}g = C_zT_fC_z^{-1}g + C_zH_fC_z^{-1}g. \label{prop2_eq}
\end{equation}
Let $x \in \beta\Omega \setminus \Omega$ and choose a net $(z_{\gamma})$ in $\Omega$ that converges to $x$. By Lemma \ref{lem1} and Lemma \ref{lem:lo_ex}, we get that $C_{z_{\gamma}}M_fC_{z_{\gamma}}^{-1} = M_{f \circ \phi_{z_{\gamma}}}$ converges strongly on $X^p_{\alpha}$ to some operator $T_x$ as $z_{\gamma} \to x$. Note that $T_x$ only depends on $x$ and not on the chosen net $(z_{\gamma})$ converging to $x$. If we define $h_x := T_x\1$, then
\begin{equation*}
\norm{f \circ \phi_{z_{\gamma}} - h_x}_{L^p_{\alpha}} = \norm{M_{f \circ \phi_{z_{\gamma}}}\1 - T_x\1}_{L^p_{\alpha}} \to 0.
\end{equation*}
As the functions $f \circ \phi_{z_{\gamma}}$ are uniformly bounded, $h_x$ is also bounded. It follows that $M_{f \circ \phi_{z_{\gamma}}} \to M_{h_x}$ and $M_{\bar{f} \circ \phi_{z_{\gamma}}} \to M_{\overline{h_x}}$ strongly on $L^p_{\alpha}$. In particular, $T_x = M_{h_x}$. By \eqref{prop2_eq} and Lemma \ref{lem1}, this implies $C_{z_{\gamma}}T_fC_{z_{\gamma}}^{-1} \to M_{h_x}$ strongly on $X^p_{\alpha}$. As $X^p_{\alpha}$ is closed, $h_x$ has to be a bounded analytic function. 
\end{proof}

\begin{definition}[{\cite[Definition 6]{FuHa}, \cite[Definition 9]{Hagger1}}]
A bounded operator $A \in \Lc(L^p_{\alpha})$ is called a band operator if there exists a positive real number $\omega$ such that $M_fAM_g = 0$ for all $f,g \in L^{\infty}(\Omega)$ with $\dist_{d_{\Omega}}(\supp f,\supp g) > \omega$. Here, $\dist_{d_{\Omega}}(M,N)$ denotes the distance between two sets $M,N \subseteq \Omega$ and is defined as
\[\dist_{d_{\Omega}}(M,N) := \inf_{x \in M,y \in N} d_{\Omega}(x,y).\]
An operator $A \in \Lc(L^p_{\alpha})$ is called band-dominated if it is the norm limit of a sequence of band operators.
\end{definition}

With a slight abuse of notation, we will call an operator $A \from X^p_{\alpha} \to L^p_{\alpha}$ or $A \from X^p_{\alpha} \to X^p_{\alpha}$ band-dominated if $AP_{\alpha} \from L^p_{\alpha} \to L^p_{\alpha}$ is band-dominated. As $P_{\alpha}$ is itself band-dominated (see \cite[Proof of Theorem 7]{FuHa}, \cite[Proof of Theorem 15]{Hagger1}) and every product of band-dominated operators is again a band-dominated operator (see \cite[Proposition 13]{FuHa}, \cite[Proposition 13]{Hagger1}), all  Toeplitz and Hankel operators with bounded symbols are also band-dominated.

\begin{proposition} \label{prop4}
Let $H \from X^p_{\alpha} \to L^p_{\alpha}$ be a band-dominated operator and suppose that $C_zHC_z^{-1} \to 0$ strongly as $z \to \partial\Omega$. Then $H$ is compact.
\end{proposition}

\begin{proof}
The main difficulty is that for bounded symmetric domains the operators $P_{\alpha}$ and $C_z$ do not commute unless $p = 2$. We therefore focus on the case of bounded symmetric domains. Recall that $C_z^{-1} = C_z$ here.

Let $\frac{1}{p} + \frac{1}{q} = 1$ and let $\tilde{C}_z \from L^q_{\alpha} \to L^q_{\alpha}$ be the reflection operator on $L^q_{\alpha}$ corresponding to $z \in \Omega$. We define $C_z^{\star} \from L^p_{\alpha} \to L^p_{\alpha}$ by
\begin{equation*}
\langle C_z^{\star}f,g \rangle_{L^2_{\alpha}} = \langle f,\tilde{C}_zg\rangle_{L^2_{\alpha}} \quad \text{for all } f \in L^p_{\alpha},\, g \in L^q_{\alpha},
\end{equation*}
that is, the adjoint of $\tilde{C}_z$. Let $(z_{\gamma})$ be a net in $\Omega$ that converges to some $x \in \beta\Omega$. The product $C_{z_{\gamma}}P_{\alpha}C_{z_{\gamma}}^{\star}$ is strongly convergent on $X^p_{\alpha}$, see \cite[Proposition 17]{Hagger2}. As $\tilde{C}_z$ is a surjective isometry, $C_z^{\star}$ is again a (surjective) isometry. Moreover, since $C_zM_fC_z = M_{f \circ \phi_z}$ and $P_{\alpha}C_zP_{\alpha} = C_zP_{\alpha}$, we have $C_z^{\star}M_fC_z^{\star} = M_{f \circ \phi_z}$ and $P_{\alpha}C_z^{\star}P_{\alpha} = P_{\alpha}C_z^{\star}$.

Now assume that $H$ is not compact. Then, since $P_{\alpha}M_{\chi_{B(0,s)}}$ is compact for every $s > 0$ (see e.g.~\cite[Proposition 15]{Hagger1}), there is an $\epsilon > 0$ such that for all $s > 0$ we have $\norm{HP_{\alpha}M_{1-\chi_{B(0,s)}}} > \epsilon$. By \cite[Proposition 21]{Hagger2}, there is a radius $r$ such that for all $s > 0$ there is a midpoint $z_s$ such that
\begin{equation} \label{eq:prop_4}
\norm{HP_{\alpha}M_{\chi_{B(z_s,r)}}} \geq \norm{HP_{\alpha}M_{1-\chi_{B(0,s)}}M_{\chi_{B(z_s,r)}}} > \frac{\epsilon}{2}.
\end{equation}
Clearly, $z_s \to \partial\Omega$ as $s \to \infty$. Therefore $(z_s)$ has a subnet, denoted by $(z_{s_\gamma})$, that converges to some $x \in \beta\Omega \setminus \Omega$. It follows
\begin{align*}
\norm{(C_{z_{s_\gamma}}HC_{z_{s_\gamma}})(C_{z_{s_\gamma}}P_{\alpha}C_{z_{s_\gamma}}^{\star})P_{\alpha}M_{\chi_{B(0,r)}}} &= \norm{HP_{\alpha}C_{z_{s_\gamma}}^{\star}M_{\chi_{B(0,r)}}C_{z_{s_\gamma}}^{\star}}\\
&= \norm{HP_{\alpha}M_{\chi_{B(z_{s_\gamma},r)}}}.
\end{align*}
As $z_{s_\gamma} \to x$, the left-hand side tends to $0$, a contradiction to \eqref{eq:prop_4}. This completes the proof for bounded symmetric domains. The second part of the proof together with $C_zP_{\alpha} = P_{\alpha}C_z$ and the obvious modifications yields a (much simpler) proof for the Fock space.
\end{proof}

We need one more preliminary lemma for our main results.

\begin{lemma} \label{lem:thm1}
Let $g \in L^p_{\alpha}$ and $z \in \Omega$. Then
\begin{equation*}
\norm{(C_z^{-1}P_{\alpha} - P_{\alpha}C_z^{-1})g}_{L^p_{\alpha}} \leq \norm{P_{\alpha}}\norm{(I-P_{\alpha})g}_{L^p_{\alpha}}.
\end{equation*}
\end{lemma}

If $p = 2$ or $\Omega = \C^n$, a direct calculation shows $P_{\alpha}C_z^{-1} = C_z^{-1}P_{\alpha}$. This lemma is therefore only relevant in case $\Omega$ is a bounded symmetric domain and $p \neq 2$. Recall that $C_z^{-1} = C_z$ in that case.

\begin{proof}
Using $P_\alpha C_z^{-1} P_\alpha = C_z^{-1} P_\alpha$ and the fact that $C_z^{-1}$ is an isometry, we have
\begin{align*}
\norm{(C_z^{-1}P_{\alpha} - P_{\alpha}C_z^{-1})g}_{L^p_{\alpha}} &\leq \norm{P_{\alpha}}\norm{(C_z^{-1}P_{\alpha} - C_z^{-1})g}_{L^p_{\alpha}}\\
&= \norm{P_{\alpha}}\norm{(I-P_{\alpha})g}_{L^p_{\alpha}}.\qedhere
\end{align*}
\end{proof}

\section{Main results}

Using the preliminary results above, we can prove the following characterization of compact Hankel operators on both Bergman and Fock spaces. Part (iii) can be further used to show that compactness is independent of $p$ and $\alpha$, see Corollary \ref{cor:thm1_2}.

\begin{theorem} \label{thm1}
Let $f \in L^{\infty}(\Omega)$ and $1<p<\infty$. The following are equivalent:
\begin{itemize}
	\item[\rm (i)] $H_f : X^p_\alpha(\Omega)\to L^p_\alpha$ is compact.
	\item[\rm (ii)] For every $g \in X^p_{\alpha}$, we have $\norm{C_zH_fC_z^{-1}g}_{L^p_{\alpha}} \to 0$ as $z \to \partial\Omega$.
	\item[\rm (iii)] For every $x \in \beta\Omega \setminus \Omega$ there is a bounded analytic function $h_x$ such that for all nets $(z_{\gamma})$ in $\Omega$ converging to $x$ we have
	\begin{equation*}
	\norm{f \circ \phi_{z_{\gamma}} - h_x}_{L^p_{\alpha}} \to 0 \quad \mbox{as} \quad z_{\gamma} \to x.
	\end{equation*}
	\item[\rm (iv)] $\norm{(I-P_{\alpha})(f \circ \phi_z)}_{L^p_{\alpha}} \to 0$ as $z \to \partial\Omega$.
	\item[\rm (v)] $\norm{H_fg_z}_{L^p_{\alpha}} \to 0$ as $z \to \partial\Omega$, where $g_z = C_z^{-1}\1$, that is, $g_z = k_{-z}$ in case $\Omega = \C^n$ and $g_z = k_z^{2/p}$ in case $\Omega$ is a bounded symmetric domain.
	\item[\rm (vi)] For every $g \in X^p_{\alpha}$ we have $\norm{H_{f \circ \phi_z}g}_{L^p_{\alpha}} \to 0$ as $z \to \partial\Omega$.
\end{itemize}
\end{theorem}

\begin{proof}
The equivalence of (i) and (ii) follows from Lemma \ref{lem1} and Proposition \ref{prop4}. By Proposition \ref{prop2}, (i) implies (iii). As $P_{\alpha}h_x = h_x$ and $P_{\alpha}$ is bounded, (iv) follows immediately from (iii).

Now assume (iv). Then
\begin{align*}
\|H_fC_z^{-1}\1\|_{L^p_{\alpha}} &= \|(I-P_{\alpha})M_fC_z^{-1}\1\|_{L^p_{\alpha}}\\
&= \|(I-P_{\alpha})C_z^{-1}M_{f \circ \phi_z}\1\|_{L^p_{\alpha}}\\
&= \|(I-P_{\alpha})C_z^{-1}(f \circ \phi_z)\|_{L^p_{\alpha}}\\
&\leq \|C_z^{-1}(I-P_{\alpha})(f \circ \phi_z)\|_{L^p_{\alpha}} + \|(C_z^{-1}P_{\alpha}-P_{\alpha}C_z^{-1})(f \circ \phi_z)\|_{L^p_{\alpha}}\\
&\leq \|(I-P_{\alpha})(f \circ \phi_z)\|_{L^p_{\alpha}} + \|P_{\alpha}\|\|(I-P_{\alpha})(f \circ \phi_z)\|_{L^p_{\alpha}}\\
&\to 0
\end{align*}
as $z \to \partial\Omega$ by Lemma \ref{lem:thm1}, hence (v).

So now assume (v). Let $(z_{\gamma})$ be a net converging to some $x \in \beta\Omega \setminus \Omega$. By Lemma \ref{lem:lo_ex}, the operator $C_{z_{\gamma}}T_fC_{z_{\gamma}}^{-1}$ converges strongly to some operator $T_x$. Set $h_x := T_x\1 \in X^p_{\alpha}$. Then
\begin{align*}
\norm{f \circ \phi_{z_{\gamma}} - h_x}_{L^p_{\alpha}} &\leq \norm{f \circ \phi_z - C_{z_{\gamma}}T_fC_{z_{\gamma}}^{-1}\1}_{L^p_{\alpha}} + \norm{C_{z_{\gamma}}T_fC_{z_{\gamma}}^{-1}\1 - h_x}_{L^p_{\alpha}}\\
&= \norm{(M_{f \circ \phi_z} - C_{z_{\gamma}}P_{\alpha}M_fC_{z_{\gamma}}^{-1})\1}_{L^p_{\alpha}} + \norm{C_{z_{\gamma}}T_fC_{z_{\gamma}}^{-1}\1 - h_x}_{L^p_{\alpha}}\\
&= \norm{C_{z_{\gamma}}(I-P_{\alpha})M_fC_{z_{\gamma}}^{-1}\1}_{L^p_{\alpha}} + \norm{C_{z_{\gamma}}T_fC_{z_{\gamma}}^{-1}\1 - h_x}_{L^p_{\alpha}}\\
&= \norm{H_fC_{z_{\gamma}}^{-1}\1}_{L^p_{\alpha}} + \norm{C_{z_{\gamma}}T_fC_{z_{\gamma}}^{-1}\1 - h_x}_{L^p_{\alpha}}\\
&\to 0
\end{align*}
as $z_{\gamma} \to x$. Now, as in the proof of Proposition \ref{prop2}, this implies $M_{f \circ \phi_{z_{\gamma}}} \to M_{h_x}$ strongly. As $h_x \in X^p_{\alpha}$, we have $(I-P_{\alpha})M_{h_x} = 0$ on $X^p_{\alpha}$. It follows
\begin{equation*}
\norm{H_{f \circ \phi_{z_{\gamma}}}g}_{L^p_{\alpha}} = \norm{(I-P_{\alpha})M_{f \circ \phi_{z_{\gamma}}}g}_{L^p_{\alpha}} \to 0
\end{equation*}
for all $g \in X^p_{\alpha}$ as $z_{\gamma} \to x$. As the net $(z_{\gamma})$ was arbitrary, this implies (vi). A simple algebraic computation shows
\begin{equation*}
C_zH_fC_z^{-1} - H_{f \circ \phi_z} = C_z(C_z^{-1}P_{\alpha} - P_{\alpha}C_z^{-1})M_{f \circ \phi_z}
\end{equation*}
and therefore (vi) implies (ii) by Lemma \ref{lem:thm1}.
\end{proof}

\begin{corollary} \label{cor:thm1_2}
Let $f \in L^{\infty}(\Omega)$. Then the compactness of $H_f : X^p_\alpha \to L^p_\alpha$ does not depend on $p$ or $\alpha$.
\end{corollary}

\begin{proof}
As $f \circ \phi_z - h_x$ is uniformly bounded in $z \in \Omega$, this follows from standard estimates (H\"older) and Theorem \ref{thm1} (iii) similarly as in \cite[Theorem 7]{StroeZhe}.
\end{proof}

In our final result we give a characterization of compact Hankel operators in terms of the classical $\VMO$-spaces, which are defined as in~\eqref{e:VMOp}. We split it into two theorems because there is a crucial difference between the two cases. In the Fock space, $H_f$ is compact if and only if $H_{\bar f}$ is compact as a consequence of Liouville's theorem. For the Bergman space, we cannot rely on Liouville and therefore need to additionally assume the compactness of $H_{\bar f}$.

\begin{theorem} \label{thm2}
If $f\in L^\infty(\Omega)$ and $\Omega = \C^n$, then {\rm (i)} to {\rm (vi)} in Theorem \ref{thm1} are further equivalent to
\begin{itemize}
	\item[\rm (vii)] $H_{\bar{f}} : F^p_\alpha \to L^p_\alpha$ is compact.
	\item[\rm (viii)] $\|f \circ \phi_z - \tilde{f}(-z)\|_{L^p_{\alpha}} \to 0$ as $z \to \partial\Omega$.
	\item[\rm (ix)] $f \in \VMO^p(\C^n) \cap L^{\infty}(\C^n)$.
\end{itemize}
\end{theorem}

\begin{proof}
Bounded analytic functions on $\C^n$ are constant by Liouville's theorem. Hence condition (iii) in Theorem \ref{thm1} is symmetric in $f$ and $\bar{f}$ and therefore $f \in \Ac$ is equivalent to (vii).

To show that $f \in \Ac$ implies (viii), take $x \in \beta\Omega$ and choose a net $(z_{\gamma})$ in $\Omega$ that converges to $x$. Since $\widetilde{f \circ \phi_z} = \tilde{f} \circ \phi_z$ for all $z \in \Omega$ and $\tilde{g} = g$ for analytic functions $g$, we get
\begin{equation*}
\lim\limits_{z_{\gamma} \to x} \tilde{f}(\phi_{z_{\gamma}}(w)) = \lim\limits_{z_{\gamma} \to x} \widetilde{f \circ \phi_{z_{\gamma}}}(w) = \widetilde{h_x}(w) = h_x(w)
\end{equation*}
for every $w \in \Omega$. In particular, $\tilde{f}(-z_{\gamma})$ converges to $h_x(0)$. As $h_x$ is a constant, Theorem \ref{thm1} (iii) implies (viii). 

Now assume (viii). As $\tilde{f}(-z)$ is a constant and $P_{\alpha}$ is bounded, we get
\begin{align*}
\norm{(I-P_{\alpha})(f \circ \phi_z)}_{L^p_{\alpha}} &\leq \norm{f \circ \phi_z - \tilde{f}(-z)}_{L^p_{\alpha}} + \norm{P_{\alpha}\left(\tilde{f}(-z) - f \circ \phi_z\right)}_{L^p_{\alpha}}\\
&\leq (1+\norm{P_{\alpha}})\norm{f \circ \phi_z - \tilde{f}(-z)}_{L^p_{\alpha}}\\
&\to 0
\end{align*}
as $z \to \partial\Omega$, hence Theorem \ref{thm1} (iv) holds.

The equivalence of (viii) and (ix) is standard and can be found in \cite[Theorem 3]{PeSchuVi} and its generalization~\cite{L18}.
\end{proof}

\begin{theorem} \label{thm3}
If $f\in L^\infty(\Omega)$ and $\Omega$ is a bounded symmetric domain, then the following are equivalent:
\begin{itemize}
	\item[\rm (vii)] $H_f, H_{\bar f}: A^p_{\alpha} \to L^p_{\alpha}$ are both compact.
	\item[\rm (viii)] $\|f \circ \phi_z - \tilde{f}(z)\|_{L^p_{\alpha}} \to 0$ as $z \to \partial\Omega$.
	\item[\rm (ix)] $f \in \VMO^p(\Omega) \cap L^{\infty}(\Omega)$.
\end{itemize}
\end{theorem}

\begin{proof}
Assuming the compactness of both $H_f$ and $H_{\bar{f}}$ again implies that $h_x$ in Theorem \ref{thm1} (iii) is constant. Thus (vii) and (viii) are equivalent by the same argument as in Theorem \ref{thm2}. The equivalence of (viii) and (ix) was proven in~\cite{MR1178032} for $\Omega = \B_n$. A similar proof works for bounded symmetric domains; we omit the details.
\end{proof}

\begin{remark}
It is actually not difficult to see that all statements in Theorem \ref{thm2} and Theorem \ref{thm3} are independent of $p$ and $\alpha$ (cf.~Corollary \ref{cor:thm1_2}). Therefore it would have been sufficient to cite the corresponding results for $p = 2$ in \cite{MR1073289} and \cite{MR0882716}. However, we decided to include short proofs in order to stress this key difference between Bergman and Fock spaces caused by Liouville's theorem.
\end{remark}

\section*{Acknowledgements}
The authors would like to thank the editor and the anonymous referees for their valuable suggestions that significantly improved the quality of the paper. This project has received funding from the European Union's Horizon 2020 research and innovation programme under the Marie Sklodowska-Curie grant agreement No 844451. Virtanen was supported in part by Engineering and Physical Sciences Research Council grant EP/T008636/1.

\begin{flushleft}
Department of Mathematics, University of Reading, Reading RG6 6AX, England\\
\textit{Email:} \texttt{r.t.hagger@reading.ac.uk} and \texttt{j.a.virtanen@reading.ac.uk}
\end{flushleft}
\end{document}